\documentclass{amsart}      
\usepackage{amssymb,amsthm, amsmath, amsfonts} 
\usepackage{graphics}                 
\usepackage{hyperref}                 
\usepackage[all]{xy}
\usepackage{enumerate}

\newtheorem{theorem}{Theorem}[section]
\newtheorem{lemma}[theorem]{Lemma}
\newtheorem{proposition}[theorem]{Proposition}

\newtheorem{corollary}[theorem]{Corollary}
\newtheorem{definition}[theorem]{Definition}
\newtheorem{example}[theorem]{Example}
\newtheorem{remark}[theorem]{Remark}
\numberwithin{equation}{section}

\DeclareMathOperator{\rad}{rad}

\DeclareMathOperator{\op}{op}

\DeclareMathOperator{\image}{Im}

\DeclareMathOperator{\Hom}{Hom}
\DeclareMathOperator{\End}{End}
\DeclareMathOperator{\Ext}{Ext}

\DeclareMathOperator{\module}{-mod}
\DeclareMathOperator{\additive}{add}
\DeclareMathOperator{\Thick}{Thick}

\newcommand{\D}{{D^b(A)}}
\newcommand{\Drb}{{D^-(A)}}
\newcommand{\K}{{K^b(_A \mathcal{P})}}
\newcommand{\Kb}{{K^b(_B \mathcal{P})}}

\newcommand{\Kr}{{K^b(_R \mathcal{P})}}
\newcommand{\Ks}{{K^b(_S \mathcal{P})}}
\newcommand{\Krb}{{K^-(_A \mathcal{P})}}
\newcommand{\PP}{{P^{\bullet}}}
\newcommand{\QQ}{{Q^{\bullet}}}

\newcommand{\T}{{\mathcal{T}}}
\newcommand{\F}{{\mathcal{F}}}

\title[Compact exceptional objects]{On compact exceptional objects in derived module categories}
\author{Liping Li}
\address{College of Mathematics and Computer Science; Performance Computing and Stochastic Information Processing (Ministry of Education), Hunan Normal University, Changsha, Hunan 410081, China.}
\email{lipingli@hunnu.edu.cn.}
\thanks{The author is supported by the National Natural Science Foundation of China 11541002, the Construct Program of the Key Discipline in Hunan Province, and the Start-Up Funds of Hunan Normal University 830122-0037. He thanks the referee for carefully checking the manuscript and providing many detailed suggestions, which make the current form much better.}
\begin{document}

\begin{abstract}
Let $A$ be a basic and connected finite dimensional algebra and $\D$ be the bounded derived category of finitely generated left $A$-modules. In this paper we consider lengths of tilting objects and indecomposable compact exceptional objects in $\D$, and prove a sufficient condition such that these lengths are bounded by the number of isomorphism classes of simple $A$-modules. Moreover, we show that algebras satisfying this criterion are bounded derived simple, and describe an algorithm to construct a family of algebras satisfying this condition.
\end{abstract}

\maketitle

\section{Introduction}

In algebraic representation theory, one of the most interesting and complicated problems is to classify algebras up to derived equivalence. Explicitly, given a finite dimensional (connected and basic) $k$-algebra $A$, we want to characterize or construct all (connected and basic) algebras whose bounded derived module categories are triangulated equivalent to the bounded derived category $\D$. This big project draws the attention of many people, and quite a few results have been obtained in two different approaches. In one direction, people have classified certain types of algebras with special properties; see \cite{AH1,AH,A,BTA,BTA1,B,Holm1,Keller}. In the other direction, several properties have been shown to be invariant under derived equivalence, such as number of isomorphism classes of simple modules and finiteness of global dimensions (\cite{Happel}), finiteness of finitistic dimensions (\cite{PX}), finiteness of strong global dimensions (\cite{HZ}), self-injective property (\cite{AR}), etc. However, there are much more questions unsolved. For instance, a finite dimensional local algebra is only derived equivalent to algebras Morita equivalent to itself (see \cite{RZ}), and we will show that path algebras of Kronecker quivers have this property as well (see Example 5.1). But a complete list of basic algebras which are only derived equivalent to algebras Morita equivalent to themselves is not available yet.

According to a fundamental result of Rickard (\cite{Rickard1, Rickard2}), an algebra $\Gamma$ is derived equivalent to $A$ if and only if there is a tilting object $T \in \D$ such that $\Gamma$ is isomorphic to the opposite algebra of $\End _{\D} (T)$. Therefore, tilting objects, and more generally, compact exceptional objects are of particular importance, and hence are extensively studied. For instance, Angeleri H\"{u}gel, Koenig, and Liu (in \cite{AKL1,AKL2,AKL3}) use them to investigate recollements and stratifications of derived categories; and Al-Nofayee and Rickard point out in \cite{AR,Rickard3} that for a fixed algebra, there are at most countably many \emph{basic} tilting objects $T$ up to isomorphism and degree shift, where $T$ is basic if its direct summands are pairwise nonisomorphic.

In this paper we mainly focus on \emph{lengths} of objects in derived categories, which are defined as follows. For an arbitrary $\PP \in \Krb$, the homotopy category of right bounded complexes of finitely generated projective $A$-modules, let
\begin{equation*}
a(\PP) = \sup \{ i \in \mathbb{Z} \mid P^i \neq 0 \}- \inf \{ i \in \mathbb{Z} \mid P^i \neq 0 \},
\end{equation*}
called the \textit{amplitude} of $\PP$ (\cite{AF}). Since $\Krb$ and the right bounded derived category $\Drb$ are equivalent as triangulated categories, for $X \in \Drb$, we define its length to be \footnote{Note that our definition of lengths is slightly different from that in \cite{HZ}. The length defined here counts terms between the first nonzero term (if it exists) and the last nonzero term, whereas the length defined in \cite{HZ} counts the number of differential maps between the first nonzero term (if it exists) and the last nonzero term. For a compact object, the difference of these two lengths is exactly 1.}
\begin{equation*}
l(X) = \inf \{ a(\PP) + 1 \mid \PP \in \Krb \text{ is quasi-isomorphic to } X \}.
\end{equation*}
Clearly, an object $X \in \Drb$ has finite length if and only if $X$ is quasi-isomorphic to a certain $\PP \in \K$, or equivalently, $X$ is compact.

Happel and Zacharia prove in \cite{HZ} that lengths of all indecomposable objects in $\D$ are bounded if and only if $A$ is \textit{piecewise hereditary}; that is, $\D$ is equivalent to $D^b(\mathcal{H})$, where $\mathcal{H}$ is a hereditary abelian category. Thus we may ask under what conditions the lengths of all indecomposable compact exceptional objects are bounded. For connected algebras, we show that this boundedness property is derived invariant (Theorem \ref{boundedness of compact exceptional objects}), and it implies the boundedness of lengths of tilting complexes (Proposition \ref{boundedness of tilting objects}). The following theorem gives us a criterion guaranteeing the above boundedness property.

\begin{theorem} \label{theorem I}
Let $A$ be a basic and connected finite dimensional algebra. For a simple $A$-module $S$, let $P_S$ be its projective cover and let $Q_S$ be the direct sum of indecomposable projective $A$-modules (up to isomorphism) not isomorphic to $P_S$. Suppose that for every simple module $S$, the socle of $P_S$ contains a simple summand $S_{\circ} \cong S$ satisfying the following conditions:
\begin{enumerate}
\item $S_{\circ}$ is not contained in the image of any homomorphism $P \to P_S$ for $P \in \additive(Q_S)$;
\item $S_{\circ}$ is contained in the kernel of any homomorphism $P_S \to P$ for $P \in \additive(Q_S)$.
\end{enumerate}
Then the lengths of tilting objects and indecomposable compact exceptional objects in $\D$ are bounded by the number of isomorphism classes of simple $A$-modules. Moreover, every indecomposable projective $A$-module appears at precisely one degree for every minimal tilting complex.
\end{theorem}

Note that the possibility that $S = S_{\circ}$ is excluded. Indeed, if it happens, then $P_S = S$, and hence by the second condition $A$ is a direct sum of two algebras, contradicting the assumption that $A$ is connected.

By \cite{AKL2} and \cite{AKL3}, an algebra $A$ is called \textit{bounded derived simple} if $\D$ has no nontrivial recollements by bounded derived module categories of algebras. Clearly, local algebras are bounded derived simple. But there are many other bounded derived simple algebras (\cite{LY1, LY2}). Algebras satisfying the conditions in the above theorem are bounded derived simple. That is:

\begin{theorem} \label{main result II}
Let $A$ be basic and connected finite dimensional algebra such that every simple $A$-module satisfies the conditions specified in Theorem \ref{theorem I}. Then the following statements are equivalent:
\begin{enumerate}
\item $\K$, the homotopy category of perfect complexes, has a nontrivial torsion pair $(\T, \F)$ such that either $\T$ or $\F$ is closed under degree shift;
\item $A$ is a triangular matrix algebra, and $\T$  (resp. $\F$) coincides with a thick subcategory generated by a projective $A$-module $P$ (resp. $Q$).
\end{enumerate}
In this case, $A$ is bounded derived simple.
\end{theorem}

The conditions in Theorem \ref{theorem I} seem artificial and mysterious, and the reader may have the intuition that algebras $A$ satisfying these conditions are complicated. Actually, these conditions imply that every vertex in the ordinary quiver of $A$ has at least one loop. Moreover, we can show that $A$ has infinite global dimension (if $A \ncong k$) and its finitistic dimension is 0. Therefore, $A$ does not have nice properties people prefer such as being hereditary, piecewise hereditary, quasi-hereditary, etc. It is reasonable to believe that indecomposable objects in $\D$ are complicated. However, since in this paper we are only interested in compact exceptional objects, the bad behaviors of $A$ contrarily restrict the size of these special objects as well as stratifications of $\D$. Furthermore, algebras satisfying these conditions are not rare. Actually, many weakly directed algebras and string algebras are examples. We will give an explicit algorithm in Section 4 to construct a big class of algebras satisfying these conditions.

The paper is organized as follow. In Section 2 we consider lengths of objects in derived module categories, and prove Theorem \ref{theorem I}. Stratification of bounded derived module categories is investigated in Section 3, where Theorem \ref{main result II} is proved. In Sections 4 and 5 we describe some applications, and ask several questions for which the answers are not clear to us.

In this paper we only consider finite dimensional algebras over an algebraically closed field $k$, although many results are still true in a much more general framework. All modules, unless specified explicitly, are finitely generated left modules. Composition of maps and morphisms is from right to left. The zero module is regarded as a trivial projective or free module. For a fixed object $X$ in a module category or a derived category, $\additive (X)$ is the additive category consisting of direct summands of finite direct sums of $X$, and $\Thick (X)$ is the smallest triangulated category (closed under isomorphisms, degree shift, direct summands, and finite coproducts) containing $X$.\footnote{We say $X$ \textit{classically generates} $\Thick (X)$ in this situation.} The degree shift functor $[-]$ in derived categories is as usually defined.

\section{Lengths of compact exceptional objects}

Throughout this section let $A$ be a basic and connected finite dimensional algebra over an algebraically closed field $k$ and let $A \module$ be the category of finitely generated left $A$-modules. Let $\K$ be the homotopy category of \textit{perfect complexes}; that is, complexes of finitely generated projective $A$-modules such that all but finitely many terms are zero. The following embedding and equivalence are well known:
\begin{equation*}
\K \subseteq \Krb \cong \Drb.
\end{equation*}
Moreover, $\K$ can be identified with a full subcategory of $\D$, and $\D$ is equivalent to $K^{-,b} (_A \mathcal{P})$, the homotopy category of right bounded chain complexes of finitely generated projective $A$-modules with finite homologies; see \cite{Happel}.

A complex $\PP \in \K$ is \textit{minimal} if it has no summands of the following form:
\begin{equation*}
\xymatrix{ \ldots \ar[r] & 0 \ar[r] & P \ar[r] ^{id} & P \ar[r] & 0 \ar[r] & \ldots},
\end{equation*}
where $P \in \additive (_AA)$. It is easy to see that $\PP$ is minimal if and only if every differential map $d_i: P^i \rightarrow P^{i+1}$ sends $P^i$ into the radical of $P^{i+1}$. An object $X$ is \textit{compact} (that is, the functor $\Hom_{\D} (X, -)$ commutes with small coproducts) if and only if it is quasi-isomorphic to a minimal perfect complex $\PP \in \K$ (which is unique up to isomorphism), and if and only if its length is finite. To calculate its length, we first choose a minimal perfect complex $\PP \in \K$ quasi-isomorphic to $X$, and let $r$ and $s$ be the degrees of the first and the last nonzero terms in this complex. Then $l(X) = l(\PP) = s - r +1$.

An object $X \in \D$ is \textit{exceptional} (or \textit{rigid}) if
\begin{equation*}
\Hom _{\D} (X, X[n]) \cong \Hom _{\Krb} (\PP, \PP[n]) = 0
\end{equation*}
whenever $n \neq 0$. An object $X \in \D$ is \textit{tilting} if it is quasi-isomorphic to a certain exceptional $\PP \in \K$ (called a \textit{tilting complex}) such that $\Thick (\PP) = \K$. Clearly, a tilting object is a direct sum of finitely many indecomposable compact exceptional objects. If all direct summands are pairwise nonisomorphic, we say that the tilting object is \textit{basic}.

\begin{proposition} \label{boundedness of tilting objects}
Let $A$ be a basic and connected algebra. If the lengths of all indecomposable compact exceptional objects are bounded, so are the lengths of tilting objects.
\end{proposition}

\begin{proof}
Let $T \in \D$ be a tilting object. Without loss of generality we can assume that $T$ is a tilting complex (i.e., $T \in \K$) and is basic since a tilting complex and the corresponding basic tilting complex obtained by taking one direct summand from each isomorphism class have the same length. Write $T =  \oplus _{i=1}^n T_i$, where $n$ is the number of isomorphism classes of simple $A$-modules and each $T_i$ is indecomposable. Clearly we can assume that every indecomposable summand is minimal. If $r, s\in \mathbb{Z}$ satisfy $T^r \neq 0 \neq T^s$ and $T^i = 0$ for $i > s$ or $i < r$, the length of $T$ is $s - r + 1$. Since the length of each indecomposable summand is bounded by a fixed number $m$, $l(T) \leqslant nm$ is also bounded if we can show that there is no gap among these indecomposable summands. That is, for any $r \leqslant j \leqslant s$, we have $T^j \neq 0$. But this is clear. Indeed, if it is not true, then we use this gap to decompose $T = T' \oplus T''$ such that
\begin{equation*}
\Hom _ {\K} (T', T'')  = 0 = \Hom _ {\K} (T'', T').
\end{equation*}
Consequently, $A$ is derived equivalent to $\End _{\K} (T') ^{\op} \oplus \End _{\K} (T'') ^{\op}$. This is impossible since connectedness is invariant under derived equivalences.
\end{proof}

By \cite{HZ}, boundedness of lengths of all indecomposable compact objects is equivalent to the piecewise hereditary property, which is invariant under derived equivalences. Similarly, boundedness of lengths of all indecomposable compact exceptional objects is invariant under derived equivalences, too.

\begin{theorem} \label{boundedness of compact exceptional objects}
Let $A$ and $B$ be two connected basic finite dimensional algebras. Suppose that $A$ and $B$ are derived equivalent. If lengths of indecomposable compact exceptional objects in $\D$ are bounded, then $B$ has the same property as well.
\end{theorem}

\begin{proof}
Without loss of generality we only consider minimal objects in homotopy categories of perfect complexes. Assume that $l (X) \leqslant m$ for all indecomposable exceptional objects $X \in \K$. Since $D^b(B)$ is derived equivalent to $\D$, by \cite{Rickard1}, there exists a tilting complex $T \in \Kb$ such that $\End _{\Kb} (T) ^{\op} \cong A$. Moreover, $T$ induces a triangulated equivalence $F: \K \rightarrow \Kb$ such that $F(A) = T$. Let $G: \Kb \rightarrow \K$ be the quasi-inverse of $F$.

Now let $Y$ be a minimal, indecomposable exceptional object in $\Kb$. Then $G (Y)$ is an indecomposable exceptional object in $\K$, which has at most $m$ nonzero terms since by our assumption $G (Y) \in \K$ is minimal and by the given condition lengths of indecomposable exceptional objects in $\K$ are bounded by $m$.

Note that $FG (Y)$ is constructed as follows (Proposition 2.10 in \cite{Rickard1}). First, for all $i \in \mathbb{Z}$, nonzero $G(Y)^i$ (which are projective $A$-modules) are replaced by objects in $\additive (T)$ to obtain a bigraded complex $Y^{\ast \ast}$ over $\additive (B)$. Then we take the total complex $Y^{\ast}$ of $Y^{\ast \ast}$ and define $FG (Y) = Y^{\ast}$. Clearly, $Y \cong FG (Y) = Y^{\ast}$, and $l (Y^{\ast}) \leqslant m + l(T)$, where $l(T)$ is the length of $T$. That is, lengths of indecomposable exceptional objects in $\Kb$ are bounded by $m + l(T)$.
\end{proof}

The following lemmas and their corollaries will be used frequently.

\begin{lemma} \label{equivalent conditions}
Let $S$, $P_S$, $S_{\circ}$ and $Q_S$ be as defined in Theorem \ref{theorem I} and suppose that the two conditions in it hold. Then every homomorphism $\alpha: P_S \to P$, where $P$ is projective, is either a split monomorphism or has $S_{\circ}$ in the kernel.
\end{lemma}

\begin{proof}
Write $P = P' \oplus Q'$ with $P' \in \additive(P_S)$ and $Q' \in \additive (Q_S)$. If $\alpha$ is not a split monomorphism, then $S_{\circ}$ is in the kernel of the induced map $P_S \to P'$. But $S_{\circ}$ also lies in the kernel of the induced map $P_S \to Q'$ by the assumption. The conclusion follows.
\end{proof}

Let $\PP \in \K$. We say that an indecomposable projective $A$-module $P$ \textit{appears at degree i} if $P^i$ has a summand isomorphic to $P$. The following lemma is crucial to prove Theorem \ref{theorem I}.

\begin{lemma} \label{orthogonal relation}
Let $S$, $P_S$, $S_{\circ}$ and $Q_S$ be as defined in Theorem \ref{theorem I} and suppose that the two conditions in it hold. If $\PP$ and $Q^{\bullet}$ are objects in $\K$ such that $r$ is the last degree of $\PP$ and $s$ the first degree of $Q^{\bullet}$ where $P_S$ appears, then $\Hom _{\K} (\PP[r], Q^{\bullet}[s]) \neq 0$.
\end{lemma}

\begin{proof}
By assumption, we have the following decompositions: $P^r \cong P_S \oplus P'$, and $P^s \cong P_S \oplus Q'$. Now we construct a chain map as follows:
\begin{equation*}
\xymatrix{ \ldots \ar[r] ^-{d_{s-3}} & P^{s-2} \ar[r] ^-{d_{s-2}} \ar[d]^0 & P^{s-1} \ar[r] ^-{d_{s-1}} \ar[d]^0 & P_S \oplus P' \ar[r] ^-{d_s} \ar[d] ^{\tilde{\alpha}} \ar@{-->}[ld] _{h_s} & P^{s+1} \ar[r] ^-{d_{s+1}} \ar[d]^0 \ar@{-->}[ld] _{h_{s+1}} & \ldots \\
\ldots \ar[r] ^-{d_{r-3}} & Q^{r-2} \ar[r] ^-{d_{r-2}} & Q^{r-1} \ar[r] ^-{d_{r-1}} & P_S \oplus Q' \ar[r] ^-{d_r} & Q^{r+1} \ar[r] ^-{d_{r+1}} & \ldots ,}
\end{equation*}
where $\tilde{\alpha} = \begin{bmatrix} \alpha & 0 \\ 0 & 0 \end{bmatrix}$, and $\alpha$ maps the top of $P_S$ to $S_{\circ}$ in the socle of $P_S$.

We check that $\tilde{\alpha} d_{s-1} = 0$ and $d_r \tilde{\alpha} = 0$. Indeed, since $\PP$ is minimal, $d_{s-1} (P^{s-1})$ is contained in the radical of $P^s$. But $\tilde{\alpha}$ maps the radical of $P^s$ to 0, so $\tilde{\alpha} d_{s-1} = 0$. Besides, by definition, $\tilde{\alpha} (P^s) = S_{\circ} \subseteq P_S \subseteq Q^r$. Since $Q^r$ is the last degree where $P_S$ appears, the map $P_S \to Q^{r+1}$ induced by $d_r$ cannot be a split monomorphism. By the previous lemma, $S_{\circ}$ is contained in the kernel of this induced map. Consequently, $d_r \tilde{\alpha} = 0$.

The chain map is not null homotopic. Otherwise, there must exist maps $h_s: P^s \rightarrow Q^{r-1}$ and $h_{s+1}: P^{s+1} \rightarrow Q^r$ such that $\tilde{\alpha} = h_sd_{r-1} + h_{s+1} d_r$. In particular, $\alpha$ is the composite of the following maps:
\begin{equation*}
\xymatrix{P_S \ar[rrr] ^-{\begin{bmatrix} h_s \iota \\ d_s \iota \end{bmatrix}} & & & Q^{r-1} \oplus P^{s+1} \ar[rrr] ^-{\begin{bmatrix} pd_{r-1} & ph_{s+1} \end{bmatrix}} & & & P_S},
\end{equation*}
where $\iota: P_S \rightarrow P^s$ and $p: Q^r \rightarrow P_S$ are the inclusion and projection respectively. But this is impossible. Indeed, since $r$ and $s$ are the first and the last degrees where $P_S$ appears, $Q^{r-1} \oplus P^{s+1}$ has no summand isomorphic to $P_S$. Therefore, by condition (1) in Theorem \ref{theorem I}, $S_{\circ} \subseteq P_S$ is not in the image of $\begin{bmatrix} pd_{r-1} & ph_{s+1} \end{bmatrix}$, and hence is not in the image of $\alpha$ by the above factorization. This contradicts the definition $\alpha$, and our claim is proved.
\end{proof}

Several immediate corollaries of this lemma are:

\begin{corollary} \label{only one degree}
Let $S$, $P_S$, $Q_S$, and $S_{\circ}$ be as defined in Theorem \ref{theorem I}, and suppose that the two conditions in it hold. Let $\PP \in \K$ be a minimal exceptional object. Then $P_S$ appears at most one degree of $\PP$.
\end{corollary}

\begin{proof}
Suppose that $P_S$ appears at more than one degrees of $\PP$. Let $r$ and $s$ be the first degree and the last degree where $P_S$ appears. By the previous lemma, $\Hom_{\K} (\PP, \PP [s-r]) \neq 0$, where $s-r \neq 0$. This contradicts the assumption that $\PP$ is exceptional.
\end{proof}

\begin{corollary} \label{only one object}
Let $A$ and $P_S$ be as in Lemma \ref{orthogonal relation}. If $\PP, \QQ \in \K$ are two minimal objects satisfying $\Hom_{\K} (\PP, \QQ[n]) = 0$ for all $n \in \mathbb{Z}$, then $P_S$ cannot appear in both $\PP$ and $\QQ$.
\end{corollary}

\begin{proof}
Suppose that $P_S$ appears in both $\PP$ and $\QQ$. Applying degree shift if necessary, we can assume that both the last degree of $\PP$ where $P_S$ appears in $\PP$ and the first degree of $\QQ$ where $P_S$ appears in $\QQ$ are 0. Thus $\Hom_{\K} (\PP, \QQ) \neq 0$ by Lemma \ref{orthogonal relation}, contradicting the assumption.
\end{proof}

Now we restate and prove Theorem \ref{theorem I}.

\begin{theorem} \label{main result I}
Let $A$ be a basic and connected finite dimensional algebra, and suppose that all simple $A$-modules satisfy the conditions in Theorem \ref{theorem I}. Then the lengths of tilting objects and indecomposable compact exceptional objects $\D$ are bounded by the number of isomorphism classes of simple $A$-modules. Moreover, every indecomposable projective $A$-module appears at precisely one degree for every minimal tilting complex.
\end{theorem}

\begin{proof}
Take an indecomposable exceptional object $\PP \in \K$. Without loss of generality we can assume that $\PP$ is minimal. Let $r, s \in \mathbb{Z}$ such that $P^r$ and $P^s$ are the first and last nonzero terms in $\PP$. Then $l(\PP) = s - r +1$ by our definition. Since $\PP$ is indecomposable, for $r \leqslant t \leqslant s$, $P^t \neq 0$. However, by Corollary \ref{only one degree}, every indecomposable projective $A$-module (up to isomorphism) appears at no more than one degree in $\PP$. Therefore, $\PP$ can have at most $n$ nonzero terms, where $n$ is the number of isomorphism classes of simple $A$-modules. That is, $l(\PP) = s - r + 1 \leqslant n$.

Let $T \in \K$ be a minimal tilting complex. By Corollary \ref{only one degree}, every indecomposable projective $A$-module appears at no more than one degree of $T$. But since $\Thick (T) = \K$, every indecomposable projective $A$-module must appear at some degree of $T$. Otherwise, suppose that $P_e \cong Ae$ does not appear, where $e \in A$ is a primitive idempotent. Then
\begin{equation*}
\Thick (T) \subseteq \Thick (A(1-e)) \neq \K.
\end{equation*}
This is impossible. Therefore, every indecomposable projective $A$-module appears at precisely one degree of $T$.

Observe that $T$ must be \textit{connected} as in the proof of Lemma \ref{boundedness of tilting objects}. That is, if $r, s \in \mathbb{Z}$ satisfy $T^r \neq 0 \neq T^s$ and $T^i = 0$ for $i>s$ or $i<r$, then $T^j \neq 0$ for $r \leqslant j \leqslant s$. Therefore, the length of $T$ must be bounded by the number of isomorphism classes of simple $A$-modules by the second statement.
\end{proof}

In the following example we use the above results to classify the derived equivalence classes of a particular algebra $A$. A similar example has been considered in \cite[Example 5.13]{AKLY}, and in Section 4 we will consider a more general construction.

\begin{example} \normalfont \label{example}
Let $A$ be the path algebra of the following quiver with relations $\delta^2 = \rho^2 = \rho \alpha = \alpha \delta = 0$.
\begin{equation*}
\xymatrix {x \ar@(ld, lu)|[] {\delta} \ar[r] ^{\alpha} & y\ar@(rd, ru)|[] {\rho}}
\end{equation*}
Up to isomorphism, this algebra has two indecomposable projective modules as follows:
\begin{equation*}
P_x = \begin{matrix} & x & \\ x & & y \end{matrix}, \qquad P_y = \begin{matrix} y \\ y \end{matrix},
\end{equation*}
both of which satisfy the conditions in Theorem \ref{theorem I}.

By Corollary \ref{only one degree}, minimal indecomposable exceptional objects in $\K$ have lengths at most 2. Moreover, if $l(\PP) = 2$, then there exists a certain $i \in \mathbb{Z}$ such that $P^i \cong P_y^{\oplus a}$ and $P^{i+1} \cong P_x ^{\oplus b}$ with $a, b \geqslant 1$. We can check that $\PP$ is indecomposable if and only if $a = b =1$. In conclusion, up to degree shift and quasi-isomorphisms, $\K$ only has three indecomposable exceptional objects: stalk complexes $P_x$ and $P_y$, and
\begin{equation*}
X := \xymatrix{P_y \ar[r] ^d & P_x},
\end{equation*}
where $d$ maps the top of $P_y$ onto the simple summand $S_y$ in the socle of $P_x$.

Using Theorem \ref{main result I}, the reader can check that up to degree shift and quasi-isomorphism $\K$ has three basic tilting complexes: $T_1 = P_x \oplus P_y$, $T_2 = P_x[-1] \oplus X$, and $T_3 = P_y \oplus X$. Clearly, $\End_{\K} (T_1) ^{\op} \cong A$.

By computation, $B^{\op} = \End_{\K} (T_2)$ is isomorphic to the path algebra of the following quiver with relations $\beta \alpha \beta = 0 = \delta^2$, $\delta \alpha = \beta \delta = 0$:
\begin{equation*}
\xymatrix {x \ar@/^.5pc/[rr] ^{\alpha} & & y\ar@(rd, ru)|[] {\delta} \ar@/^.5pc/[ll] ^{\beta}}.
\end{equation*}
Similarly, by computation, $C^{\op} = \End_{\K} (T_3)$ is isomorphic to the path algebra of the following quiver with relations $\alpha \beta \alpha = 0 = \delta^2$, $\delta \alpha = \beta \delta = 0$:
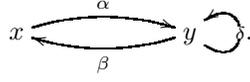
\begin{equation*}
\xymatrix {x \ar@/^.5pc/[rr] ^{\alpha} & & y \ar@/^.5pc/[ll] ^{\beta} \ar@(rd, ru)|[] {\delta}}.
\end{equation*}
We conclude that up to Morita equivalence, $A$ is derived to three algebras: $A$, $B$, and $C$ as they lie in different Morita equivalence classes.

A careful observation tells us that $B \cong C^{\op}$. This is reasonable since $A \cong A^{\op}$. Moreover, $B$ has a tilting complex $T = Be_x [1] \oplus M$ where $M \cong Be_y /Be_x$. It is easy to check that $\End_B (T) ^{\op} \cong A$.
\end{example}

\begin{proposition}
Let $A$ be a basic and connected finite dimensional algebra, and suppose that all simple $A$-modules satisfy the conditions in Theorem \ref{theorem I}. Then the finitistic dimension of $A$ is 0.
\end{proposition}

Recall that the \emph{finitistic dimension} of $A$ is the supremum of projective dimensions of indecomposable objects in $A\module$ with finite projective dimension. The finitistic dimension of $A$ equals 0 if and only if all finitely generated $A$-modules having finite projective dimension are projective.

\begin{proof}
Suppose that the conclusion is wrong. Then we can find some $M \in A\module$ such that the projective dimension of $M$ is $n$ with $n > 0$. Take a minimal projective resolution of $M$ as follows
\begin{equation*}
0 \rightarrow P^n \rightarrow P^{n-1} \rightarrow \ldots \rightarrow P^0 \rightarrow 0.
\end{equation*}
By our assumption, $P^n \neq 0$, and the map $d_n: P^n \rightarrow P^{n-1}$ must be injective.

Take a nonzero indecomposable summand $P_S$ of $P^n$ and denote by $\iota$ the inclusion $P_S \rightarrow P^n$. Since the projective resolution is minimal, $\iota$ cannot be a split monomorphism. By Lemma \ref{equivalent conditions}, $\iota$ and hence $d_n$ cannot be injective. The conclusion follows by the contradiction.
\end{proof}

\section{Stratification of bounded derived module categories}

In this section we consider stratification of bounded derived module categories of algebras satisfying the condition specified in Theorem \ref{theorem I}, showing that they are bounded derived simple. The following definition is taken from \cite[Section 2, Chapter I]{BR}.

\begin{definition}
A pair of strict full subcategories $(\T, \F)$ of a triangulated category $\mathcal{C}$ is called a torsion pair if the following conditions hold:
\begin{enumerate}
\item $\mathcal{C} (T, F) = 0$ for any $T \in \T$ and $F \in \F$.
\item $T \in \T$ implies $T[1] \in \T$; and $F \in \F$ implies $F[-1] \in \F$.
\item For any $X \in \mathcal{C}$, there is a triangle (which is unique up to isomorphism)
\begin{equation*}
\xymatrix{T_X \ar[r] & X \ar[r] & F_X \ar[r] & T_X[1]}.
\end{equation*}
\end{enumerate}
\end{definition}

Note that $\T$ is the \textit{left perpendicular category} of $\F$, and $\F$ is the \textit{right perpendicular category} of $\T$ (see \cite{GL}).

The following key observation is crucial to prove the main result of this section.

\begin{lemma} \label{orthogonal categories}
Let $A$ be a finite dimensional algebra. Let $e$ and $f$ be two orthogonal idempotents such that $e + f = 1$. If there exists a torsion pair $(\T, \F)$ of $\K$ such that $\T \subseteq \Thick (Ae)$, $\F \subseteq \Thick (Af)$, then $\T = \Thick (Ae)$ and $\F = \Thick (Af)$.
\end{lemma}

\begin{proof}
We only need to show the other inclusions. Take an arbitrary object $U \in \Thick (Af)$. Since $(\T, \F)$ is a torsion pair of $\K$, there is a canonical triangle
\begin{equation*}
\xymatrix{V \ar[r] & U \ar[r] & W \ar[r] & V[1]}
\end{equation*}
with $V \in \T \subseteq \Thick (Ae)$ and $W \in \F \subseteq \Thick (Af)$. However, since both $U$ and $W$ are contained in $\Thick (Af)$, so is $V$. Therefore, $V \in \Thick (Ae) \cap \Thick (Af)$. We claim that $V$ is quasi-isomorphic to 0. If this holds, then $U \cong W \in \F$, so $\Thick (Af) \subseteq \F$.

Indeed, since $V \in \Thick (Ae)$, we get a minimal representation of $V$
\begin{equation*}
\PP: \quad \ldots \rightarrow 0 \rightarrow P^r \rightarrow \ldots \rightarrow P^s \rightarrow 0 \rightarrow \ldots
\end{equation*}
such that all $P^i$ are contained in $\additive (Ae)$. Similarly, since $V \in \Thick (Af)$, we get another minimal representation of $V$
\begin{equation*}
Q^{\bullet}: \quad \ldots \rightarrow 0 \rightarrow Q^l \rightarrow \ldots \rightarrow Q^t \rightarrow 0 \rightarrow \ldots
\end{equation*}
such that all $Q^i$ are contained in $\additive (Af)$.

Now let us regard $V$ and its two representations as objects in $D^b(A)$. If $V \neq 0$, then there is a simple module $X$ and a certain $n \in \mathbb{Z}$ with $\Hom _{D^b(A)} (V, X[n]) \neq 0$. Consequently, we have both $\Hom _{D^b(A)} (P^{\bullet}, X[n]) \neq 0$ and $\Hom _{D^b(A)} (Q^{\bullet}, X[n]) \neq 0$. However, this impossible since terms in $P^{\bullet}$ and $Q^{\bullet}$ have no isomorphic indecomposable summands. Therefore, $V$ is quasi-isomorphic to 0, and our claim is proved.

We have shown $\F \supseteq \Thick (Af)$. The fact that $\T \supseteq \Thick (Ae)$ can be shown by the same argument. This finishes the proof.
\end{proof}

An immediate corollary is:

\begin{corollary}
Let $\T, \F, e, f$ be as in the previous proposition. Then $A$ is isomorphic to the triangular matrix algebra (see \cite{Li4})
\begin{equation*}
\begin{bmatrix} eAe & 0 \\ fAe & fAf \end{bmatrix}.
\end{equation*}
\end{corollary}

\begin{proof}
Since $Ae \in \T$, $Af \in \F$, and $(\T, \F)$ is a torsion pair of $\K$, one should have
\begin{equation*}
eAf \cong \Hom_A (Ae, Af) \cong \Hom_{\K} (Ae, Af) = 0.
\end{equation*}
\end{proof}

The following proposition implies the first part of Theorem \ref{main result II}.

\begin{proposition}
Let $A$ be a basic and connected finite dimensional algebra, and suppose that all simple $A$-modules satisfy the conditions in Theorem \ref{theorem I}. If $(\T, \F)$ is a torsion pair of $\K$ such that either $\T$ or $\F$ is closed under degree shift, then it coincides with a torsion pair $(\Thick (P), \Thick (Q))$ induced by two projective $A$-modules $P$ and $Q$.
\end{proposition}

\begin{proof}
let $E$ be a complete set of primitive orthogonal idempotents of $A$, and define
\begin{align*}
& E_1 = \{ \epsilon \in E \mid A\epsilon \text{ appears in a certain minimal } X \in \T \};\\
& E_2 = \{ \epsilon \in E \mid A\epsilon \text{ appears in a certain minimal } Y \in \F \}.
\end{align*}
Define $e = \sum _{\epsilon \in E_1} \epsilon$ and $f = \sum _{\epsilon \in E_2} \epsilon$. Since $(\T, \F)$ is a torsion pair of $\K$ and either $\T$ or $\F$ is closed under shifts, on one hand $\Hom _{\K} (X, Y[n]) = 0$ for all $n \in \mathbb{Z}$, $X \in \T$ and $Y \in \F$, so $E_1 \cap E_2 = \emptyset$ by Corollary \ref{only one object}; on the other hand, $E_1 \cup E_2 = E$ since objects in $\T$ and $\F$ classically generated $\K$. Consequently, $e + f = 1$. Thus by Lemma \ref{orthogonal categories} $\T = \Thick (Ae)$ and $\F = \Thick (Af)$.
\end{proof}

Now we define recollements of triangulated categories; for more details, see \cite{AKL1, AKL2, AKL3, AKLY, CX, Koenig, NS}.

\begin{definition}
A recollement of a triangulated category $\mathcal{C}$ by triangulated categories $\mathcal{D}$ and $\mathcal{E}$ is expressed diagrammatically as follows
\begin{equation*}
\xymatrix{\mathcal{D} \ar[rr] ^{i_{\ast}} & & \mathcal{C} \ar[rr] ^{j^{\ast}} \ar@/_1.5pc/[ll] _{i^{\ast}} \ar@/^1.5pc/[ll] _{i^!} & & \mathcal{E} \ar@/^1.5pc/[ll] _{j_{\ast}} \ar@/_1.5pc/[ll] _{j_!}}
\end{equation*}
with six exact, additive triangulated functors $i^{\ast}, i_{\ast}, i^!, j^!, j^{\ast}, j_{\ast}$ satisfying the following conditions:
\begin{enumerate}
\item $(i^{\ast}, i_{\ast}, i^!)$, and $(j_!, j^{\ast}, j_{\ast})$ both are adjoint triples;
\item $i_{\ast}, j_!$ and $j_{\ast}$ are fully faithful;
\item $i^! j_{\ast} = 0$;
\item for each $X \in \mathcal{C}$, there are triangles
\begin{align*}
\xymatrix{i_{\ast}i^! (X) \ar[r] & X \ar[r] & j_{\ast} j^{\ast} (X) \ar[r] & i_{\ast}i^! (X)[1]},\\
\xymatrix{j_!j^{\ast} (X) \ar[r] & X \ar[r] & i_{\ast} i^{\ast} (X) \ar[r] & j_!j^{\ast} (X)[1].}
\end{align*}
\end{enumerate}
\end{definition}

The following theorem, proved in \cite{AKL2}, gives a relation between recollements of derived module categories and exceptional compact objects.

\begin{theorem}
\cite[Corollary 2.5]{AKL2} Let $\Gamma$ be a finite dimensional algebra. If the bounded derived category $D^b(\Gamma)$ is a recollement of $D^b(B)$ and $D^b(C)$, then there are objects $T_1, T_2 \in D^b(\Gamma)$ satisfying:
\begin{enumerate}
\item Both $T_1$ and $T_2$ are compact and exceptional;
\item $\Hom _{D^b(\Gamma)} (T_1, T_2[n]) = 0$ for all $n \in \mathbb{Z}$;
\item $T_1 \oplus T_2$ generates $D^b(\Gamma)$ as triangulated category. \footnote{Here we say $T_1 \oplus T_2$ generates $D^b(\Gamma)$ if for any nonzero object $X \in D^b(\Gamma)$, there is a certain $n \in \mathbb{Z}$ such that $\Hom_{D^b (\Gamma)} (T_1 \oplus T_2, X[n]) \neq 0$. This holds if $\Thick (T_1 \oplus T_2) = \K$; i.e., $T_1 \oplus T_2$ classically generates $\K$.}
\end{enumerate}
Moreover, we have $\End _{D^b(\Gamma)} (T_1) ^{\op} \cong C$ and $\End _{D^b(\Gamma)} (T_2) ^{\op} \cong B$.
\end{theorem}

According to \cite{LY1}, an algebra $\Gamma$ is said to be \textit{bounded derived simple} if $D^b(\Gamma)$ has no nontrivial recollements by bounded derived module categories of algebras.

We restate the second part of Theorem \ref{main result II}.

\begin{theorem}
Let $A$ be a basic and connected finite dimensional algebra, and suppose that all simple $A$-modules satisfy the conditions in Theorem \ref{theorem I}. Then $A$ is bounded derived simple.
\end{theorem}

\begin{proof}
This is clearly true if $A$ is a local algebra. Suppose that $A$ is not local and is not bounded derived simple. Therefore, by \cite[Theorem 5.12]{AKLY}, $A$ is not $\K$-simple either. That is, the homotopy category $\K$ has a nontrivial recollement as follows
\begin{equation*}
\xymatrix{\Ks \ar[rr] ^{i_{\ast}} & & \K \ar[rr] ^{j^{\ast}} \ar@/_1.5pc/[ll] _{i^{\ast}} \ar@/^1.5pc/[ll] _{i^!} & & \Kr \ar@/^1.5pc/[ll] _{j_{\ast}} \ar@/_1.5pc/[ll] _{j_!}},
\end{equation*}
where $R$ and $S$ are finite dimensional algebras. It is clear that $i_{\ast} S$, $j_! R$, and $j_{\ast} R$ are all compact. Without loss of generality we can assume that they are minimal. Since $\Ks = \Thick (S)$ and $i_{\ast}$ is a full embedding, we have $\image i_{\ast} = \Thick (i_{\ast} S)$. Similarly, $\image j_! = \Thick (j_! R)$, and $\image j_{\ast} = \Thick (j_{\ast} R)$.

Let $E$ be a complete set of primitive orthogonal idempotents of $A$, and let $X$ (resp. $Y$ and $Z$) be a minimal object in $K^b (_A \mathcal{P})$ isomorphic to $j_! R$ (resp. $i_{\ast} S$ and $j_{\ast} R$). Define
\begin{align*}
& E_1 = \{ \epsilon \in E \mid A\epsilon \text{ appears in } X \},\\
& E_2 = \{ \epsilon \in E \mid A\epsilon \text{ appears in } Y \},\\
& E_3 = \{ \epsilon \in E \mid A\epsilon \text{ appears in } Z \}.
\end{align*}
Correspondingly, let
\begin{equation*}
e = \sum _{\epsilon \in E_1} \epsilon, \quad f = \sum _{\epsilon \in E_2} \epsilon, \quad \lambda = \sum _{\epsilon \in E_3} \epsilon.
\end{equation*}

By \cite[Lemma 2.6]{CX}, $(\Thick (j_!R), \Thick (i_{\ast} S))$ is a torsion pair of $\K$. By Proposition \ref{orthogonal categories} and its proof, we know
\begin{equation*}
e + f = 1, \quad \Thick (j_!R) = \Thick (Ae), \quad \Thick (i_{\ast}S) = \Thick (Af).
\end{equation*}
Therefore, $eAf \cong \Hom_A (Ae, Af) = 0$.

By \cite[Lemma 2.6]{CX}, $(\Thick (i_{\ast} S), \Thick (j_{\ast} R))$ is a torsion pair of $\K$, too. Again, by Proposition \ref{orthogonal categories} and its proof, we know
\begin{equation*}
f + \lambda = 1, \quad \Thick (j_{\ast} R) = \Thick (A\lambda).
\end{equation*}
Consequently,
\begin{equation*}
\lambda = e, \quad \Thick (A \lambda) = \Thick (Ae) \quad \Rightarrow \quad fAe \cong \Hom_A (Af, Ae) = 0.
\end{equation*}
But this implies that $A \cong eAe \oplus fAf$ since $eAf = 0 = fAe$, contradicting the assumption that $A$ is connected. The conclusion follows.
\end{proof}

However, the algebras $A$ in the above theorem in general are not \textit{derived simple}. Indeed, the algebra $A$ in Example \ref{example} is bounded derived simple by the above theorem, but it is not derived simple. Actually, for every triangular matrix algebra $A = (eAe, fAf, fAe)$, as we pointed out in \cite{Li4} $D(A)$ always has the following nontrivial recollement
\begin{equation*}
\xymatrix{D(fAf) \ar[rr] ^{i_{\ast}} & & D(A) \ar[rr] ^{j^{\ast}} \ar@/_1.5pc/[ll] _{i^{\ast}} \ar@/^1.5pc/[ll] _{i^!} & & D(eAe) \ar@/^1.5pc/[ll] _{j_{\ast}} \ar@/_1.5pc/[ll] _{j_!}},
\end{equation*}
where the functors are specified as in Proposition 3.6 of \cite{Li4}.

\section{Weakly directed algebras}

In this section we apply the results in the previous sections to a special class of algebras. As before, we only consider connected and basic finite dimensional algebras over an algebraically closed field $k$.

\begin{definition}
A finite dimensional algebra $A$ is said to be weakly directed if there is a complete set of primitive orthogonal idempotents $E = \{ e_i \} _{i=1}^n$ such that $\Hom_A (Ae_j, Ae_i) \cong e_j A e_i \neq 0$ only if $j \geqslant i$.
\end{definition}

Since the field is algebraically closed, $A$ is weakly directed if and only if oriented cycles in the ordinary quiver are all loops. Examples of weakly directed algebras include quotient algebras of finite dimensional hereditary algebras, local algebras, category algebras of skeletal finite EI categories (\cite{Li1}), extension algebras of standard modules of standardly stratified algebras (\cite{Li2}).

The following lemma is described in \cite{Li3}.

\begin{lemma}\label{simple modules}
Let $A$ be a weakly directed algebra. Then every simple $A$-module can be identified with a simple $eAe$-module for some $e \in E$. Moreover, for any $M \in A\module$ and $f \in E$, the left $fAf$-module $fM$ is also an $A$-module.
\end{lemma}

\begin{proof}
The first statement is contained in Proposition 2.2 in \cite{Li3}. To show the second one, note that by the weakly directed structure of $A$, all arrows between different vertices generate a two-sided ideal $J$ of $A$, and $\oplus_{i=1}^n e_iAe_i \cong A/J$. As a summand, $fAf$ is a quotient algebra of $A$. The second statement follows from this observation.
\end{proof}

By this lemma, the $A$-module $fM$ and the $fAf$-module $fM$ have the same composition factors, and hence the same Loewy length.

The proposition below gives a practical method to check the conditions in Theorem \ref{theorem I} for weakly directed algebras.

\begin{lemma} \label{criteria for directed algebras}
Let $A$ be a weakly directed algebra and $e \in E$. Then the simple $A$-module $S_e \cong Ae / \rad Ae$ satisfies the conditions in Theorem \ref{theorem I} if and only if the socle of $Ae$ contains a simple summand $S_e \cong Ae / \rad Ae$ such that for any $x \in eA(1-e)$ we have $S_e x = 0$.
\end{lemma}

By Lemma \ref{simple modules}, $Ae / \rad Ae  \cong eAe / \rad eAe$, and $S_e \subseteq eAe$. Therefore, the right action of $x \in eA(1-e)$ on $S_e$ is induced from the multiplication of $eAe$ by $eA(1-e)$ from the right side.

\begin{proof}
We first prove the if part by checking conditions in Theorem \ref{theorem I}. Note that the trace of $Q_S = A(1-e)$ is $A(1-e)Ae$, which is not supported on $e$ since $eA(1-e) Ae = 0$ by the weakly directed structure. But $S_e \subseteq eAe$ is not contained in $(1-e)Ae$, so condition (1) in Theorem \ref{theorem I} is also true.

Let $\beta: Ae \rightarrow P$ be an $A$-module homomorphism, where $P \in \additive (A(1-e))$. Since $\Hom_A (Ae, A(1-e)) \cong eA(1-e)$, $\beta$ corresponds to an element $x = \beta(e) \in eP$, so $\beta(Ae) = Aex$. By the assumption, $S_e x = 0$. Therefore, $\beta(S_e) = 0$. In other words, $S_e$ is contained in the kernel of $\beta$, and hence condition (2) in Theorem \ref{theorem I} holds.

Conversely, suppose that conditions in Theorem \ref{theorem I} hold. Then by the previous paragraph, for any $x \in eA(1-e)$, we can find a module homomorphism $\beta \in \Hom_A (Ae, A(1-e))$ such that $\beta (Ae) = Aex$. In particular, $S_ex = \beta(S_e) = 0$ since by Lemma \ref{equivalent conditions} $S_e$ is in the kernel of $\beta$. This finishes the proof.
\end{proof}

\begin{remark}\normalfont
Since for any $e \in E$, $eA(1-e)$ is a finite dimensional vector space, we can take a finite basis $\{ x_i \} _{i=1}^n$ of $eA(1-e)$. Moreover, we can choose a minimal set of elements $\{v_j\} _{j=1}^m \subseteq eAe$ which generates the socle of $eAe$. Therefore, in practice we only need to verify that $v_j x_i = 0$ for all $1\leqslant i\leqslant n$ and $1 \leqslant j \leqslant m$.
\end{remark}

Consider the following example:

\begin{example}\normalfont
Let $A$ be the path algebra of the following quiver with relations $\delta^2 = \theta^2 = \rho^2 = 0$, $\theta \rho = \rho \theta = 0$, $\alpha \delta = \beta \delta = 0$, and $\theta \alpha = \rho \beta = 0$.
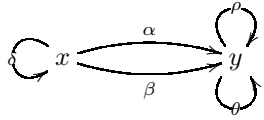
\begin{equation*}
\xymatrix{ x \ar@/^.5pc/[rr] ^{\alpha} \ar@/_.5pc/[rr] _{\beta} \ar@(lu,ld)[]|{\delta} & & y \ar@(dl,dr)[]|{\theta} \ar@(ul,ur)[]|{\rho}}
\end{equation*}
Indecomposable projective modules are:
\begin{equation*}
P_x = \begin{matrix} & x & \\ x & y & y \\ & y & y \end{matrix}, \qquad P_y = \begin{matrix} & y & \\ y & & y \end{matrix}.
\end{equation*}

It is routine to check that $P_x$ satisfies all conditions in Theorem \ref{theorem I}, but the indecomposable projective module $P_y$ fails condition (2). Indeed, there is a short exact sequence
\begin{equation*}
\xymatrix{ 0 \ar[r] & P_y \ar[r] & P_x \ar[r] & M \ar[r] & 0},
\end{equation*}
where $M = \begin{matrix} & x & \\ x & & y \end{matrix}$. This is because the socle of $P_y$ is spanned by $\rho$ and $\theta$, but $\theta \beta \neq 0$ and  $\rho \alpha \neq 0$.
\end{example}

We end this section by introducing a way to construct weakly directed algebras for which all simple modules satisfy the conditions in Theorem \ref{theorem I}. Let $Q=(Q_0, Q_1)$ be a finite connected quiver without oriented cycles, where $Q_0$ and $Q_1$ are the vertex set and the arrow set respectively. To each $v \in Q_0$ we assign an integer $m_v \geqslant 2$, and to each arrow $\alpha: v \to w$ we assign an integer $l_{\alpha} \geqslant 1$ such that $l_{\alpha} < \min \{m_v, m_w\}$.

Now add a loop $t_v$ to each vertex $v \in Q_0$ to get another quiver $\tilde{Q}$. Define
\begin{equation*}
R = k \tilde{Q} / \langle t_v ^{m_v}; \, t_w \alpha - \alpha t_v; \, \alpha t_v ^{l_{\alpha}} \mid v \in Q_0; \, Q_1 \ni \alpha: v \to w \rangle.
\end{equation*}
This is a finite dimensional algebra containing a two-sided ideal $J$ generated by all arrows in $Q_1$. Let $I \subseteq J^2$ be an arbitrary two-sided ideal, and define $A = R/I$.

Here is an example explaining the above construction.

\begin{example}\normalfont
Let $A$ be the path algebra of the following quiver with relations
\begin{enumerate}
\item $\delta^3 = \rho^3 = \theta^3 = 0$;
\item $\alpha \delta = \rho \alpha$, $\alpha \delta^2 = 0$;
\item $\beta \rho = \theta \beta$, $\beta \rho^2 = 0$;
\item $\beta \alpha = 0$.
\end{enumerate}
\begin{equation*}
\xymatrix{x \ar[r]^{\alpha} \ar@(lu,ld)[]|{\delta} & y \ar[r]^{\beta} \ar@(ul,ur)[]|{\rho} & z \ar@(ru,rd)[]|{\theta}}
\end{equation*}

This algebra can be constructed as follows. Take the quiver $x \rightarrow y \rightarrow z$ and assign
\begin{equation*}
m_x = m_y = m_z = 3; \quad l_{\alpha} = l_{\beta} = 2.
\end{equation*}
Define
\begin{equation*}
R = k\tilde{Q} / \langle t_x^3, \, t_y^3, \, t_z^3; \, \alpha t_y - t_y \beta, \, \beta t_y - t_z \beta; \, t_y \alpha^2, \, t_z\beta^2 \rangle
\end{equation*}
Finally we let $I = \langle \beta \alpha \rangle$ be the two-sided ideal of $R$. Then $A \cong R/I$.
\end{example}

The main result of this section is:

\begin{proposition}
Let $A$ be an algebra constructed as above. Then every simple $A$-module satisfies the conditions in Theorem \ref{theorem I}. Consequently, if $X \in \D$ is a tilting object or an indecomposable compact exceptional object, then its length is bounded by $|Q_0|$. Moreover, $A$ is bounded derived simple.
\end{proposition}

\begin{proof}
By Lemma \ref{criteria for directed algebras}, it suffices to check that for every vertex $v \in Q_0$,
\begin{enumerate}
\item The socle $S_v$ of $e_vAe_v$ is contained in the socle of $Ae_v$;
\item for every $w \in Q_0$ with $w \neq v$, one has $S_v A e_w = 0$.
\end{enumerate}
We reminder the reader of our construction,
\begin{equation*}
I \subseteq J \text{ and } A = R / I \quad \Rightarrow \quad e_v A e_v = e_v R e_v.
\end{equation*}
In particular, if the socle $S_v$ of $e_vRe_v$, which is isomorphic to $e_v R e_v / \rad e_vRe_v$, is contained in socle of $Re_v$, then it is contained in the socle of $Ae_v$ as well since the quotient map $R \to R/I = A$ induces an obvious $R$-module surjection $Re_v \to Ae_v$. Furthermore, since every $x \in e_v A e_w$ is the image of some $\hat{x} \in e_v R e_w$, clearly $S_v R e_w = 0$ implies $S_v A e_w = 0$. With this observation, one only needs to check the above two conditions for $R$.

Note that $e_v R e_v = k[t_v] / (t_v^{m_v})$, so $S_v$ is the one dimensional space spanned by $t_v ^{m_v - 1}$. Since $\rad R$ is generated by arrows $\alpha \in Q_1$ and loops $t_w$, $w \in Q_0$, and for each arrow $\alpha \in Q_1$ starting at $v$ or the loop $t_v$ one has
\begin{equation*}
t_v S_v = k t_v t_v^{m_v - 1} = 0, \quad \alpha t_v^{m_v - 1} = \alpha t_v ^{l_{\alpha}} t_v ^{m_v - l_{\alpha} - 1} = 0,
\end{equation*}
we conclude that $(\rad R) \cdot S_v = 0$, so $S_v$ is in the socle of $Re_v$. This proves (1).

To check (2), one just observe that
\begin{equation*}
S_v R e_w = S_v e_v R e_w = k t_v ^{m_v - 1} e_vRe_w
\end{equation*}
and $e_v R e_w$ is contained in the left ideal generated by arrows $\beta \in Q_1$ ending at $v$. But for any such $\beta$, one has
\begin{equation*}
S_v \beta = k t_v^{m_v - 1} \beta = t_v ^{m_v - l_{\beta} - 1} t_v ^{l_{\beta}} \beta = 0.
\end{equation*}
\end{proof}

\section{Questions and remarks}

In this section let $A$ be an arbitrary finite dimensional algebra. If there are only finitely many basic tilting complexes in $\D$ up to isomorphism and degree shift, then $A$ is only derived equivalent to finitely many basic algebras up to isomorphism. In particular, the lengths of tilting objects are bounded. Unfortunately, the converse statement is not true. For instance, let $A$ be a hereditary algebra of infinite representation type. Then there exist infinitely many pairwise nonisomorphic tilting modules. Indeed, since $A$ is of infinite representation type, there are infinitely many indecomposable modules in a preprojective component of the Auslander-Reiten quiver of $A$. Take an arbitrary indecomposable $A$-module $M$ from this component. It is well known that $M$ is a partial tilting module. By Bongartz's lemma, we can always complete $M$ to a basic tilting module. Since each basic tilting module only has finitely many indecomposable summands, we conclude that there are infinitely many pairwise nonisomorphic tilting modules. However, it is not guaranteed that these basic tilting modules will produce infinitely many pairwise nonisomorphic basic algebras, as shown in the following example.

\begin{example}\normalfont
Let $A$ be the path algebra of the Kronecker quiver with two vertices $x$ and $y$, and $n \geqslant 2$ arrows from $x$ to $y$. Let $\Gamma$ be a basic algebra derived equivalent to $A$. Since the number of isomorphism classes of simple modules is derived invariant, $\Gamma$ has two simple modules up to isomorphism. Moreover, $\Gamma$ is piecewise hereditary, so it is directed. Therefore, $\Gamma$ is isomorphic to the path algebra of a directed quiver with two vertices $u$ and $v$, and $m$ arrows from $u$ to $v$. Note that the characteristic polynomial of the Coxeter transformation of $A$ and $\Gamma$ must be the same; see \cite{Lenzing}. This happens if and only if $m = n$ by computation. Therefore, $A$ is only derived equivalent to algebras Morita equivalent to itself. But $A$ is hereditary and is of infinite representation type, so it has infinitely many pairwise nonisomorphic basic tilting modules.
\end{example}

The following result would not be surprising to the reader at all. Actually, the classifications up to derived equivalence has been obtained. For details, please refer to \cite{AH,HRS,Keller}.

\begin{proposition}
Any connected hereditary algebra of finite representation type has only finitely many basic tilting complexes up to isomorphism and degree shift. Consequently, it is derived equivalent to finitely many basic algebras up to isomorphism.
\end{proposition}

\begin{proof}
Let $T$ be a basic tilting complex in $\D$ where $A$ is connected and hereditary and has finite representation type. Note that the length of every indecomposable exceptional object in $\K$ is bounded by $2$. Therefore, the length of $T$ is bounded by $2n$, where $n$ is the number of isomorphism classes of simple $A$-modules. But it is clear that $T$ corresponds to sequence
\begin{equation*}
T \cong M_1 [r_1] \oplus M_2 [r_2] \oplus \ldots \oplus M_n[r_n], \quad r_1 \leqslant r_2 \leqslant \ldots \leqslant r_n,
\end{equation*}
where all $M_i$ are indecomposable $A$-modules. Therefore, $r_n - r_1 \leqslant 2n$.\footnote{Actually, $r_{i+1} - r_i \leqslant 1$ for $1 \leqslant i \leqslant n-1$ since $\Ext_A^2 (-, -) = 0$. Otherwise, as we did in the proof of the previous proposition, we can show that $A$ is derived equivalent to an algebra which is not connected. This is impossible.} Since $A$ is of finite representation type, the number of such sequences of length at most $2n$ is finite. That is, there are only finitely many basic tilting complexes up to degree shift and isomorphism.
\end{proof}

In Section 2 we proved that if the lengths of all indecomposable compact exceptional objects are bounded, so are the lengths of all tilting complexes. We wonder whether the converse statement is true. This is not obvious since Bongartz's lemma does not hold in derived categories; see \cite{Rickard1}. Therefore, given an indecomposable compact exceptional object $T$ in $\D$, people do not know whether there must exist another compact exceptional object $T' \in \D$ such that $T \oplus T'$ is a basic tilting complex. It is also not clear to the author whether there exists an algebra $A$ for which the lengths of tilting objects are bounded, but $A$ is derived equivalent to infinitely many algebras in different Morita equivalence classes. It would be interesting to describe some concrete examples, or show that it is not possible.

\end{document}